\documentclass{amsart}
\usepackage{bbm}
\usepackage{}

\usepackage{amsfonts}
\usepackage{amsfonts,amssymb,amsmath,amsthm}
\usepackage{mathrsfs}
\usepackage{graphicx}
\usepackage{url}
\usepackage{enumerate}

\newtheorem{theorem}{Theorem}[section]
\newtheorem{lemma}[theorem]{Lemma}

\newtheorem{corollary}[theorem]{Corollary}
\theoremstyle{definition}
\newtheorem{definition}[theorem]{Definition}
\newtheorem{remark}{Remark}

\numberwithin{equation}{section}

\DeclareMathOperator{\diam}{diam}

\DeclareMathOperator{\V}{V}
\DeclareMathOperator{\vol}{Vol}\DeclareMathOperator{\A}{A}

\author{Lixia Yuan}
\address{
 Department of Mathematics\\
  Fudan University\\
  Shanghai, China}
\email{yuan\underline{ }lixia@foxmail.com}
\author{Wei Zhao}
\address{
Department of Mathematics\\
East China University of Science and Technology\\
Shanghai, China}
\email{szhao\underline{ }wei@yahoo.com}

\begin{document}

\title[]{Cheeger's constant and the first eigenvalue of a closed Finsler manifold}

\begin{abstract}
In this paper, we consider Cheeger's constant and the first eigenvalue of the nonlinear Laplacian on a closed Finsler manifold.  A Cheeger type inequality and a Buser type inequality are established for closed Finsler manifolds. As an application, we obtain a Finslerian version of Yau's lower estimate for the first eigenvalue.
\end{abstract}
\maketitle

\section{Introduction}
The study of the eigenvalues of Laplacian is a classical and important problem in
Riemannian geometry, which highlights the interplay of the geometry-topology of the manifold with the analytic properties of functions.
In order to bound below
the first eigenvalue $\lambda_1(M)$ of a closed Riemannian manifold $(M^n,g)$, Cheeger\cite{C} introduced Cheeger's constant
\[
\mathbbm{h}(M):=\underset{\Gamma}\inf \frac{\A(\Gamma)}{\min\{\vol(D_1),\vol(D_2)\}},
\]
where $\Gamma$ varies over compact $(n-1)$-dimensional submanifolds
of $M$ which divide $M$ into disjoint open submanifolds $D_1$, $D_2$
of $M$ with common boundary $\partial D_1=\partial D_2=\Gamma$, and he proved that
\[
\lambda_1(M)\geq \frac{\mathbbm{h}^2(M)}{4}.\tag{1.1}\label{1.1}
\]
Moreover, even if $M$ has a boundary, Cheeger's inequality still holds if
$\lambda_1(M)$ is subject to the Neumann boundary condition or the Dirichlet boundary condition\cite{C2,SY}. This inequality has found a number of applications, e.g., \cite{B1,C2,G}.

It is an important result due to Buser\cite{B} that $\mathbbm{h}(M)$ is actually equivalent to $\lambda_1(M)$, with constants depending only on the
dimension and the Ricci curvature of $M$. More precisely, if $\mathbf{Ric}\geq -(n-1)\delta^2$, then
\[
\lambda_1(M)\leq C(n)(\delta \mathbbm{h}(M) +\mathbbm{h}^2(M)).\tag{1.2}\label{1.2}
\]
We refer to \cite{C,C2,L,BPP} for more details of these two inequalities.

Finsler geometry is just Riemannian geometry without quadratic restriction. However, there are many Laplacians on a Finsler manifold, e.g., \cite{AL}. Among them, an important one was introduced by Shen\cite{Sh3}, which is obtained by a canonical energy functional on the Sobolev space and is exactly the Laplace-Beltrami operator if the Finsler metric is Riemannian. This Laplacian has a close relationship with curvatures and plays a crucial role in establishing the comparison theorems for Finsler manifolds\cite{Sh1,Sh2,WX,ZY}, but it is quasilinear and dependent on the measure. While
the measure on a Finsler manifold can be defined in various ways and essentially different
results may be obtained, e.g., \cite{AlB,AlT}.
In general, the eigenfunctions of this Laplacian are not smooth but $C^{1,\alpha}$\cite{GS}. Hence, it seems indeed difficult to compute the first eigenvalue even for a the Eucildean sphere $\mathbb{S}^n$ equipped with a Randers metric $F=\alpha+\beta$, where $\alpha$ is the canonical Riemannian metric and $\beta$ is a $1$-form on $\mathbb{S}^n$.
The purpose of this paper is to investigate the relationship between Cheeger's constant and the first eigenvalue of such Laplacian.

Let $(M,F,d\mu)$ be a closed Finsler $n$-manifold, where $d\mu$ be any measure on $M$. According to \cite{Sh1,Sh3}, the first nontrivial eigenvalue $\lambda_1(M)$ is defined as
\[
\lambda_1(M)=\underset{f\in \mathscr{H}_0(M)\backslash\{0\}}{\inf}\frac{\int_M F^{*2}(df)d\mu}{\int_M|f|^2d\mu},
\]
where $F^*$ is the dual metric of $F$ and $\mathscr{H}_0(M):=\{f\in W^1_2(M):\,\int_M fd\mu=0\}$. It follows from \cite{GS,Sh1,Sh3} that $\lambda_1(M)$ is
the smallest positive eigenvalue of Shen's Laplacian.
Inspired by \cite{C3,Sh1,Sh3}, we define Cheeger's constant of a closed Finsler manifold as
\[
\mathbbm{h}(M)=\underset{\Gamma}{\inf}\frac{\min\{\A_\pm(\Gamma)\}}{\min\{\mu(D_1),\mu(D_2)\}},
\]
where $\Gamma$ varies over compact $(n-1)$-dimensional submanifolds
of $M$ which divide $M$ into disjoint open submanifolds $D_1$, $D_2$
of $M$ with common boundary $\partial D_1=\partial D_2=\Gamma$ and $\A_\pm(\Gamma)$ denote the areas of $\Gamma$ induced by the outward and inward normal vector fields $\mathbf{n}_\pm$. In general, $\A_+(\Gamma)\neq \A_-(\Gamma)$. In fact, one can construct examples in which the ratio of these two areas can be arbitrarily lager (see \cite{C3}).
Recall that the reversibility $\lambda_F$ of $F$ is defined by $\lambda_F:=\sup_{y\neq 0}F(y)/F(-y)$ (cf. \cite{R}). Clearly, $\lambda_F\geq 1$. If $\lambda_F=1$, $F$ is called a reversible metric, in which case $\A_+(\cdot)=\A_-(\cdot)$.

First, we have the following Cheeger type inequality.
\begin{theorem}\label{Th1}Let $(M,F,d\mu)$ be a closed Finsler manifold with the reversibility $\lambda_F\leq \lambda$. Then
\[
\lambda_1(M)\geq\frac{\mathbbm{h}^2(M)}{4\lambda^2} .
\]
\end{theorem}
It should be remarkable that if $M$ has a boudary, Chen\cite{C3} proved the above inequality still holds if $\lambda_1$ subject to the Dirichlet boundary condition. Recall
the uniform constant\cite{E} $\Lambda_F$ is defined as
\[
\Lambda_F:=\underset{X,Y,Z\in TM\backslash\{0\}}{\sup}\frac{g_Z(X,X)}{g_Z(Y,Y)}.
\]
 $\Lambda_F\geq1$ with equality if and only if $F$ is Riemannian.
As an application of Theorem \ref{Th1}, we obtain the following theorem, which is a Finslerian version of Yau's lower estimate for the first eigenvalue\cite{Y}.
\begin{theorem}Let $(M,F,d\mu)$ be a closed Finsler manifold, where $d\mu$ is either the Busemann-Hausdorff measure or the Holmes-Thompson measure.
 Then $\lambda_1(M)$ can be bounded from below in terms of the diameter, volume, uniform constant and a lower bound for the Ricci curvature.
\end{theorem}

Moreover, we also have a Buser type inequality for Finsler manifolds.
\begin{theorem}\label{Th2}
Let $(M,F,d\mu)$ be a closed Finsler $n$-manifold with the Ricci curvature $\mathbf{Ric}\geq -(n-1)\delta^2$ and the uniform constant $\Lambda_F\leq \Lambda$. Then
\[
\lambda_1(M)\leq C(n,\Lambda)\left(\delta {\mathbbm{h}}(M)+{\mathbbm{h}^2}(M)\right).
\]
\end{theorem}
In the Riemannian case, $\lambda_F=\Lambda_F=1$. Hence, Theorem 1.1 implies (\ref{1.1}) while Theorem 1.3 implies (\ref{1.2}). In particular, for a Randers metric $F=\alpha+\beta$, the uniform constant $\Lambda_F=(1+b)^2(1-b)^{-2}$, where $b:=\sup_{x\in M}\|\beta\|_\alpha(x)$ (see Corollary \ref{Randers1} below). For the Busemann-Hausdorff measure or the Holmes-Thompson measure, the S-curvature of a Berwald metric always vanishes (see Theorem \ref{Berwald1} below). Then we have the following corollary.
\begin{corollary}Let $(M,F,d\mu)$ be a closed Finsler $n$-manifold with the Ricci curvature $\mathbf{Ric}\geq -(n-1)k^2$.

\noindent (1) If $F=\alpha+\beta$, then $\lambda_1(M)\leq C(n,b)\left(\delta {\mathbbm{h}}(M)+{\mathbbm{h}^2}(M)\right)$.

\noindent (2) If $F$ is a Berwald metric, then $\lambda_1(M)\leq C(n,\lambda_F)\left(\delta {\mathbbm{h}}(M)+{\mathbbm{h}^2}(M)\right)$.

\end{corollary}

\section{Preliminaries}
In this section, we recall some definitions and properties about Finsler manifolds. See \cite{BCS,Sh1} for more details.

Let $(M,F)$ be a (connected) Finsler $n$-manifold with Finsler metric $%
F:TM\rightarrow [0,\infty)$.
Let $(x,y)=(x^i,y^i)$ be local coordinates on $%
TM$. Define
\begin{align*}
g_{ij}(x,y):=\frac12\frac{\partial^2 F^2(x,y)}{\partial y^i\partial y^j}, \ G^i(y):=\frac14 g^{il}(y)\left\{2\frac{\partial g_{jl}}{\partial x^k}(y)-\frac{\partial g_{jk}}{\partial x^l}(y)\right\}y^jy^k,
\end{align*}
where $G^i$ is the geodesic coefficients. A smooth curve $\gamma(t)$ in $M$ is called a (constant speed) geodesic if it satisfies
\[
\frac{d^2\gamma^i}{dt^2}+2G^i\left(\frac{d\gamma}{dt}\right)=0.
\]
Define the Ricci curvature by
$\mathbf{Ric}(y):=\overset{n}{\underset{i=1}{\sum}}R^i_{\,i}(y)$,
where
\[
R^i_{\,k}(y):=2\frac{\partial G^i}{\partial x^k}-y^j\frac{\partial^2G^i}{\partial x^j\partial y^k}+2G^j\frac{\partial^2 G^i}{\partial y^j \partial y^k}-\frac{\partial G^i}{\partial y^j}\frac{\partial G^j}{\partial y^k}.
\]
Set $S_xM:=\{y\in T_xM:F(x,y)=1\}$ and $SM:=\cup_{x\in M}S_xM$. The reversibility $\lambda_F$ and the uniformity constant $\Lambda_F$ of $(M,F)$ are defined by
\[
\lambda_F:=\underset{y\in SM}{\sup}F(-y),\ \Lambda_F:=\underset{X,Y,Z\in SM}{\sup}\frac{g_X(Y,Y)}{g_Z(Y,Y)}.
\]
Clearly, ${\Lambda_F}\geq \lambda_F^2\geq 1$. $\lambda_F=1$ if and only if $F$ is reversible, while $\Lambda_F=1$ if and
only if $F$ is Riemannian.

The dual Finsler metric $F^*$ on $M$ is
defined by
\begin{equation*}
F^*(\eta):=\underset{X\in T_xM\backslash \{0\}}{\sup}\frac{\eta(X)}{F(X)}, \ \
\forall \eta\in T_x^*M.
\end{equation*}
The Legendre transformation $\mathfrak{L} : TM \rightarrow T^*M$ is defined
by
\begin{equation*}
\mathfrak{L}(X):=\left \{
\begin{array}{lll}
& g_X(X,\cdot) & \ \ \ X\neq0, \\
& 0 & \ \ \ X=0.%
\end{array}
\right.
\end{equation*}
In particular, $F^*(\mathfrak{L}(X))=F(X)$. Now let $f : M \rightarrow \mathbb{R}$ be a smooth function on $M$. The
gradient of $f$ is defined by $\nabla f = \mathfrak{L}^{-1}(df)$. Thus we
have $df(X) = g_{\nabla f} (\nabla f,X)$.

Let $d\mu$ be a measure on $M$. In a local coordinate system $(x^i)$,
express $d\mu=\sigma(x)dx^1\wedge\cdots\wedge dx^n$. In particular,
the Busemann-Hausdorff measure $d\mu_{BH}$ and the Holmes-Thompson measure $d\mu_{HT}$ are defined by
\begin{align*}
&d\mu_{BH}=\sigma_{BH}(x)dx:=\frac{\vol(\mathbb{B}^{n})}{\vol(B_xM)}dx^1\wedge\cdots\wedge dx^n,\\ &d\mu_{HT}=\sigma_{HT}(x)dx:=\left(\frac1{\vol(\mathbb{B}^{n})}\int_{B_xM}\det g_{ij}(x,y)dy^1\wedge\cdots\wedge dy^n \right) dx^1\wedge\cdots\wedge dx^n,
\end{align*}
where $B_xM:=\{y\in T_xM: F(x,y)<1\}$.

For $y\in T_xM\backslash\{0\}$, define the distorsion of $(M,F,d\mu)$ as
\begin{equation*}
\tau(y):=\log \frac{\sqrt{\det g_{ij}(x,y)}}{\sigma(x)}.
\end{equation*}
And S-curvature $\mathbf{S}$ is defined by
\begin{equation*}
\mathbf{S}(y):=\frac{d}{dt}[\tau(\dot{\gamma}(t))]|_{t=0},
\end{equation*}
where $\gamma(t)$ is the geodesic with $\dot{\gamma}(0)=y$.

\section{A Cheeger type inequality for Finsler manifolds}
\begin{definition}[\cite{GS,Sh3}]
Let $(M,F,d\mu)$ be a compact Finsler manifold. Denote
$\mathscr{H}_0(M)$ by
\[
\mathscr{H}_0(M):= \left \{
\begin{array}{lll}
&\{f\in W^1_2(M):\,\int_M fd\mu=0\}, &\partial M=\emptyset,\\
\\
&\{f\in W^1_2(M):\,f|_{\partial M}=0\},&\partial M\neq\emptyset.
\end{array}
\right.
\]

Define the canonical energy functional $E$ on
$\mathscr{H}_0(M)\backslash\{0\}$ by
\[
E(u):=\frac{\int_M F^*(du)^2d\mu}{\int_M u^2 d\mu}.
\]

$\lambda$ is an eigenvalue of $(M,F,d\mu)$ if there is a function
$u\in\mathscr{H}_0(M)-\{0\}$ such that $d_u E=0$ with
$\lambda=E(u)$. In this case, $u$ is called an eigenfunction
corresponding to $\lambda$. The first eigenvalue of $(M,F,d\mu)$,
$\lambda_1(M)$, is defined by
\[\lambda_1(M):=\underset{u\in
\mathscr{H}_0(M)\backslash\{0\} }{\inf}E(u),
\]
which is the smallest positive critical value of $E$.
\end{definition}
\begin{remark}
It should be noticeable that $u$ is an eigenfunction
corresponding to $\lambda$ if and only if
\[
\Delta u +\lambda u=0 \text{ (in the weak sense)},
\]
where $\Delta:=\text{div}\circ \nabla$ is the Laplacian induced by Shen\cite{GS,Sh1,Sh3}.
\end{remark}

Let
$i:\Gamma{\hookrightarrow}M$ be a smooth hypersurface
embedded in $M$. For each $x\in\Gamma$, there exist two 1-forms
$\omega_\pm(x)\in T^*_xM$ satisfying
$i^*(\omega_\pm(x))=0$ and $F^*(\omega_\pm(x))=1$. Then $\textbf{n}_\pm(x):=\mathfrak{L}^{-1}(\omega_\pm(x))$ are two unit normal
vectors on $\Gamma$. In general, $\textbf{n}_-\neq -\textbf{n}_+$(see \cite{C3}). Let
$d\A_\pm$ denote the (area) measures induced by $\textbf{n}_\pm$, i.e.,
$d\A_\pm=i^*(\textbf{n}_\pm\rfloor d\mu)$.

\begin{definition}\label{def2}
Let $(M,F,d\mu)$ be a closed Finsler manifold. Cheeger's constant
${\mathbbm{h}}(M)$ is defined by
\[
\mathbbm{h}(M)=\underset{\Gamma}{\inf}\frac{\min\{\A_\pm(\Gamma)\}}{\min\{\mu(D_1),\mu(D_2)\}},
\]
where $\Gamma$ varies over compact $(n-1)$-dimensional submanifolds
of $M$ which divide $M$ into disjoint open submanifolds $D_1$, $D_2$
of $M$ with common boundary $\partial D_1=\partial D_2=\Gamma$.
\end{definition}

To prove Theorem \ref{Th1}, we need the following co-area formula.
\begin{theorem}[\cite{Sh1}]\label{Coar}
Let $(M,F,d\mu)$ be a Finsler manifold and $\phi$ is piecewise
smooth function with compact support. Then for any continuous
function $f$,
\[
\int_M f
d\mu=\int^\infty_{-\infty}\left[\int_{\phi^{-1}(t)}f\frac{d\A_{\mathbf{n}}}{F(\nabla\phi)}\right]dt,
\]
where $\mathbf{n}:=\nabla\phi/F(\nabla\phi)$.
\end{theorem}

Theorem \ref{Coar} then yields the following lemma.
\begin{lemma}\label{chele}
For all positive function $f\in C^\infty(M)$, we have
\begin{align*}
(1)&\int_0^\infty\mu(\Omega(t))dt=\int_Mfd\mu,\\
(2)&\,{\mathbbm{h}}(M)\int^\infty_{0}\min\{\mu(\Omega(t)),\mu(M)-\mu(\Omega(t))\}dt\leq \int_MF^*(df)d\mu,
\end{align*}
where $\Omega(t):=\{x\in M:f(x)\geq t\}$.
\end{lemma}

\begin{proof}Without loss of generality, we assume that $f$ is nonconstant. For almost every $t\in[\min f, \max f]$, $\Omega(t)$ is a domain
in $M$, with compact closure and smooth boundary.
Note that $\mathbf{n}:={\frac{\nabla f}{F(\nabla f)}}$ is a unit normal vector field along $\partial\Omega(t)$.

\noindent (1). It follows Theorem \ref{Coar} that
\[
\frac{d}{dt}\mu(\Omega(t))=-\int_{f^{-1}(t)}\frac{d\A_{\mathbf{n}}}{F(\nabla f)}.
\]
Thus, we have
\begin{align*}
\int^\infty_0\mu(\Omega(t))dt=-\int^\infty_0td\mu(\Omega(t))=\int^\infty_0tdt \int_{f^{-1}(t)}\frac{d\A_{\mathbf{n}}}{F(\nabla f)} =\int_Mf d\mu.
\end{align*}

\noindent (2).
Theorem \ref{Coar} now yields
\begin{align*}
\int_MF^*(df)d\mu=\int^\infty_{0}\A_{\mathbf{n}}(\partial\Omega(t))dt\geq {\mathbbm{h}}(M)\int^\infty_{0}\min\{\mu(\Omega(t)),\mu(M)-\mu(\Omega(t))\}dt.
\end{align*}

\end{proof}

\begin{proof}[\textbf{Proof of Theorem \ref{Th1}}]Given a smooth function $f$ on $M$, let $\alpha$ be a median of $f$, i.e.,
\[
\mu(\{x: f(x)\geq \alpha\})\geq \frac12\mu(M), \ \mu(\{x: f(x)\leq \alpha\})\geq \frac12\mu(M).
\]
Set $f_+:=\max\{f-\alpha,0\}$ and
$f_-:=\min\{f-\alpha,0\}$. By the definition of median, one can check that for any $t>0$,
\[
\mu(\{x: f_+^2(x)\geq t\})\leq \frac12\mu(M), \ \mu(\{x: f_-^2(x)\geq t\})\leq \frac12\mu(M).
\]
Thus, the above inequalities together with Lemma \ref{chele} yield
\begin{align*}
&{\mathbbm{h}}(M)\int_M|f-\alpha|^2d\mu={\mathbbm{h}}(M)\int_M(f^2_++f^2_-)d\mu\\
=&{\mathbbm{h}}(M)\left(\int_0^\infty\mu(\{x: f_+^2(x)\geq t\})dt+\int_0^\infty\mu(\{x: f_-^2(x)\geq t\})dt\right)\\
\leq& \int_MF^*(df^2_+)d\mu+\int_MF^*(df^2_-)d\mu=2\int_Mf_+F^*(df_+)+(-f_-)F^*(-df_-)d\mu\\
\leq &2\lambda \int_M|f-\alpha|F^*(df)d\mu\leq 2\lambda \left(\int_M|f-\alpha|^2d\mu\right)^{\frac12}\left(\int_MF^{*2}(df)d\mu\right)^{\frac12}.
\end{align*}
Hence,
\[
\int_MF^{*2}(df)d\mu\geq \frac{{\mathbbm{h}^2}}{4\lambda^2}\int_M|f-\alpha|^2d\mu.
\]
Since $\int_Mfd\mu=0$,
\[
\underset{\alpha\in \mathbb{R}}{\inf}\int_M|f-\alpha|^2d\mu \geq \int_Mf^2d\mu.
\]
\end{proof}

By a Croke type isoperimetric inequality\cite{YZ}, one has the following result. Also refer to \cite[Theorem 6.2, Prposition 6.4]{ZY} for a reversible version of the isoperimetric inequality.
\begin{theorem}[\cite{YZ,ZY}]\label{croke}Let $(M,F,d\mu)$ be a closed Finlser manifold with $\mathbf{Ric}\geq (n-1)k$, where $d\mu$ denotes either the Busemann-Hausdorff measure or the Holmes-Thompson measure. Then
\[
\mathbbm{h}(M)\geq \frac{(n-1)\mu(M)}{2\vol(\mathbb{S}^{n-2})\Lambda_F^{4n+\frac12}\diam (M)\int^{\diam (M)}_0\mathfrak{s}_k^{n-1}(t)dt},
\]
where $\diam(M)$ denotes the diameter of $M$.
\end{theorem}

Theorem \ref{croke} together with Theorem \ref{Th1} now yields a Finslerian version of Yau's lower estimate for the first eigenvalue\cite{Y}.
\begin{theorem}\label{Yau}
Let $(M,F,d\mu)$ be a closed Finlser manifold with $\mathbf{Ric}\geq (n-1)k$, where $d\mu$ denotes either the Busemann-Hausdorff measure or the Holmes-Thompson measure. Then
\[
\lambda_1(M)\geq \left(\frac{(n-1)\mu(M)}{4\vol(\mathbb{S}^{n-2})\Lambda_F^{4n+1}\diam (M)\int^{\diam (M)}_0\mathfrak{s}_k^{n-1}(t)dt}\right)^2.
\]
That is, $\lambda_1(M)$, of a closed Finsler manifold, can be bounded from below in terms of the diameter, volume, uniform constant and a lower bound for the Ricci curvature.
\end{theorem}

\section{Volume comparison}
In this section, we will study the properties of the polar coordinate system of a Finsler manifold, which is useful to show Theorem \ref{Th2}. Refer to \cite{Sh2,ZY} for more details.

Let $(M,F,d\mu)$ be a forward complete Finsler $n$-manifold. In the rest of this paper, we always assume that $d\mu$ is either the Busemann-Hausdorff measure or the Homles-Thompson measure. Given $p\in M$, denote by $(r,y)=(r,\theta^\alpha)$, $1\leq \alpha\leq n$, the polar coordinates about $p$. Express
\[
d\mu=\hat{\sigma}(r,y)dr\wedge d\nu_p(y),
\]
where $d\nu_p$ is the measure on $S_pM$ induced by $F$.

\begin{lemma}\label{le2}
Let $(M,F,d\mu)$ be as above. If $\mathbf{Ric}\geq (n-1)k$ ($k\leq 0$)
and $\Lambda_F\leq \Lambda$, then
\begin{align*}
(1)\ &{\hat{\sigma}(\min\{i_y,r\},y)}\geq \Lambda^{-2n}\frac{\A_{n,k}(r)}{\V_{n,k}(R)-\V_{n,k}(r)}{\int_r^R\hat{\sigma}(\min\{i_y,t\},y)dt}, \ \forall\, 0<r\leq R;\\
(2)\ &{\int_{r_0}^{r_1}\hat{\sigma}(\min\{i_y,t\},y)dt}\geq   \Lambda^{-2n}\frac{\V_{n,k}(r_1)-\V_{n,k}(r_0)}{\V_{n,k}(r_2)-\V_{n,k}(r_1)}{\int_{r_1}^{r_2}\hat{\sigma}(\min\{i_y,t\},y)dt}, \\ &\forall\, 0<r_0<r_1<r_2;\\
(3)\ &{\int_0^r\hat{\sigma}(\min\{i_y,t\},y)dt}\geq
\Lambda^{-2n}\frac{\V_{n,k}(r)}{\V_{n,k}(R)}{\int_0^R\hat{\sigma}(\min\{i_y,t\},y)dt},\ \forall\, 0<r\leq R.
\end{align*}
Here, $\V_{n,k}(r)$ (resp. $\A_{n,k}(r)$) is the volume (resp. area) of ball (resp. sphere) with radius $r$ in the Riemannian space form of constant curvature $k$, that is,
\[
\A_{n,k}(r)=\vol(\mathbb{S}^{n-1})\mathfrak{s}^{n-1}_k(r),\ \V_{n,k}(r)=\vol(\mathbb{S}^{n-1})\int^r_0\mathfrak{s}^{n-1}_k(t)dt.
\]
\end{lemma}
\begin{proof}It is easy to check that $\Lambda^{-n}\leq e^{-\tau(y)}\leq \Lambda^n$, for all $y\in SM$.
By \cite{Sh2,ZY}, for each $y\in S_pM$, we have
\[
\frac{\partial}{\partial r}\left(\frac{\hat{\sigma}_p(r,y)}{e^{-\tau(\gamma_y(r))}\mathfrak{s}^{n-1}_k(r)}\right)\leq 0, \ 0<r<i_y,
\]
which implies
\[
\frac{\partial}{\partial
r}\left(\frac{\hat{\sigma}_p(\min\{r,i_y\},y)}{e^{-\tau(\gamma_y(\min\{r,i_y\}))}\mathfrak{s}^{n-1}_k(r)}\right)\leq
0,\ \text{a.e. } r>0.
\]

Hence,
\[
\frac{\hat{\sigma}_p(\min\{r,i_y\},y)}{\hat{\sigma}_p(\min\{R,i_y\},y)}\geq
\frac{e^{-\tau(\gamma_y(\min\{r,i_y\}))}\mathfrak{s}^{n-1}_k(r)}{e^{-\tau(\gamma_y(\min\{R,i_y\}))}\mathfrak{s}^{n-1}_k(R)}\geq
\Lambda^{-2n}\frac{\mathfrak{s}^{n-1}_k(r)}{\mathfrak{s}^{n-1}_k(R)},\
0< r\leq R.
\]
Then (1), (2) follows. And (3) follows from Gromov's lemma\cite[Lemma 3.1]{Ch3}.
\end{proof}

Note that $d\mathfrak{A}_+|_{(r,y)}:=\hat{\sigma}(r,y)d\nu_p(y)$ is the measure on $S^+_p(r)$ induced by $\nabla r$. Then we have the following result.
\begin{lemma}\label{le3}
Let $i: \Gamma \hookrightarrow  M$ be a hypersurface. If the reversibility $\lambda_F\leq \lambda$, then
$d\A_{\pm}|_{(r,y)}\geq \lambda^{-1}d\mathfrak{A}_+|_{(r,y)}$, for any point $(r,y)=x\in  \Gamma$ ($r>0$).
\end{lemma}
\begin{proof} Let
$\textbf{n}$ denote a unit normal vector field on $\Gamma$. Thus,
\[
d\mathfrak{A}_+=|i^*(\nabla r\rfloor d\mu)|=|g_{\textbf{n}}(\textbf{n},\nabla r)|d\A_{\textbf{n}}\leq \lambda d\A_{\textbf{n}}.
\]
\end{proof}

\section{A Buser type isoperimetric inequality for starlike domains}

In this section, we will extend Buser's isoperimetric inequality\cite[Lemma 5.1]{B}
 to Finsler setting. However, the original method of Buser's cannot be used directly, since the Finsler metrics considered here can be nonreversible. To overcome this difficulty, we introduce the "reverse" of a Finsler metric.  The {\it reverse} of a Finsler metric $F$ is defined by $\tilde{F}(y):=F(-y)$.

A direct calculation yields the following two lemmas.
\begin{lemma}\label{rele1}
For each $y\neq 0$,  we have
\[
\tilde{G}^i(y)=G^i(-y),\, \widetilde{{\mathbf{Ric}}}(y)=\mathbf{Ric}(-y),
\]
where $\tilde{G}^i$ (resp. $G^i$) is the spray of $\tilde{F}$ (resp.
$F$) and $\widetilde{{\mathbf{Ric}}}$ (resp. $\mathbf{Ric}$) is the Ricci
curvature of $\tilde{F}$ (resp. $F$). Hence, if $\gamma$ is a
geodesic of $F$, then the reverse of $\gamma$ is a geodesic of
$\tilde{F}$.
\end{lemma}

\begin{lemma}\label{rele2}
Let $(M,F)$ be a Finsler manifold. Then
$d\tilde{\mu}=d\mu$,
where $d\tilde{\mu}$ (resp. $d\mu$) denotes the Busemann-Hausdorff
measure or the Holmes-Thompson measure of $\tilde{F}$ (resp. $F$).
\end{lemma}

\begin{corollary}\label{reco}
Let $(M,F,d\mu)$ be a Finsler manifold and $\Gamma$ be a smooth
hypersurface embedded in $M$. Thus $\tilde{\A}_\pm(\Gamma)=\A_\mp(\Gamma)$,
where $d\tilde{\A}$ (resp. $d\A$) denote the induced measure on $\Gamma$ by $d\tilde{\mu}$ (resp. $d\mu$).
\end{corollary}
\begin{proof}
Let $\textbf{n}_\pm$ (resp. $\tilde{\textbf{n}}_\pm$) be the unit normal vector along $\Gamma$
in $(M,F)$ (resp. $(M,\tilde{F})$). It is easy to check that $\tilde{\textbf{n}}_\pm=-\textbf{n}_\mp$.
Then we are done by Lemma \ref{rele2}.
\end{proof}

From above, we obtain the following key lemma.
\begin{lemma}\label{infact}
Let $D$ be a star-like domain (with respect to $p$) in $M$
with $B^+_p(r)\subset D\subset B^+_p(R)$. Let $\Gamma$ be a smooth hypersurface embedded in $D$ which divides $D$ into disjoint open sets $D_1$, $D_2$ in $D$ with common boundary $\partial D_1=\partial D_2=\Gamma$. Suppose $\mathbf{Ric}\geq (n-1)k$ ($k<0$) and $\Lambda_F\leq \Lambda$. If $\mu(D_1\cap
B^+_p(r/(2\sqrt{\Lambda})))\leq \frac12 \mu(B^+_p(r/(2\sqrt{\Lambda})))$, then
\begin{align*}
\frac{\A_\pm(\Gamma)}{\mu(D_1)}\geq \underset{0<\beta<\frac{r}{2\sqrt{\Lambda}}}{\max}\left\{\frac{\A_{n,k}(\beta)\left[\V_{n,k}\left(\frac{r}{2\sqrt{\Lambda}}\right)-\V_{n,k}(\beta)\right]}{2\Lambda^{4n+\frac12}\V_{n,k}(r)\V_{n,k}(R)}\right\},
\end{align*}
\end{lemma}
\begin{proof}For convenience, set $B(\rho):=B^+_p(\rho)$, for any $\rho>0$.
Clearly, $\mu(D_1\cap B(r/(2\sqrt{\Lambda})))\leq \mu(D_2\cap
B(r/(2\sqrt{\Lambda})))$. Let $\alpha\in (0,1)$ be a constant which
will be chosen later. 

\noindent \textbf{Step 1}: Suppose $\mu(D_1\cap
B(r/(2\sqrt{\Lambda})))\leq \alpha \mu(D_1)$.

For each $q\in
D_1-\text{Cut}_p$, set $q^*$ is the last point on the minimal
geodesic segment $\gamma_{pq}$ from $p$ to $q$, where this ray intersects
$\Gamma$. If the whole segment $\gamma_{pq}$ is contained in $D_1$, set
$q^*:=p$.

Fix a positive number $\beta\in (0,r/(2\sqrt{\Lambda}))$. Let $(t,y)$ denote the polar coordinate system about $p$. Given a point $q=(\rho,y)\in D_1-\text{Cut}_p-B(r/(2\sqrt{\Lambda}))$, set
\[
\text{rod}(q):=\{(t,y): \beta\leq t\leq \rho\}.
\]
Define
\begin{align*}
&\mathcal {D}^1_1:=\{q\in D_1-\text{Cut}_p-\overline{B(r/(2\sqrt{\Lambda}))}: q^* \notin B(\beta)\};\\
&\mathcal {D}^2_1:=\{q\in D_1-\text{Cut}_p-\overline{B(r/(2\sqrt{\Lambda}))}: \text{rod}(q)\subset D_1\};\\
&\mathcal {D}^3_1:=\{q\in B(r/(2\sqrt{\Lambda}))-\overline{B(\beta)}: \exists\, x\in \mathcal {D}^2_1, \text{such that } q\in \text{rod}(x)\}.
\end{align*}
By Lemma \ref{le2}, we obtain that
\[
\frac{\mu(\mathcal {D}^3_1)}{\mu(\mathcal {D}^2_1)}\geq \Lambda^{-2n}\frac{\V_{n,k}(r/(2\sqrt{\Lambda}))-\V_{n,k}(\beta)}{\V_{n,k}(R)-\V_{n,k}(r/(2\sqrt{\Lambda}))}=:\gamma^{-1}.
\]
It follows from the assumption that
\[
(1-\alpha)\mu(D_1)\leq \mu(D_1-B(r/(2\sqrt{\Lambda}))) , \ \mu(\mathcal {D}^3_1)\leq \mu(D_1\cap B(r/(2\sqrt{\Lambda})))\leq \alpha \mu(D_1).
\]
Note that $D_1-B(r/(2\sqrt{\Lambda}))\subset \mathcal {D}^1_1 \cup \mathcal {D}^2_1$. From above, we have
\begin{align*}
(1-\alpha)\mu(D_1)\leq \mu(\mathcal {D}^2_1)+\mu(\mathcal {D}^1_1)\leq \gamma \alpha \mu(D_1)+\mu(\mathcal {D}^1_1).\tag{5.1}\label{5.1}
\end{align*}

Set $\mathscr{D}_1^1:=\{y\in S_pM: \,\exists\, t>0, \text{ such that } (t,y)\in \mathcal {D}^1_1\}$.
Clearly,
\[
\mu(\mathcal {D}^1_1)=\int_{\mathscr{D}_1^1}d\nu_p(y)\int_{r/(2\sqrt{\Lambda})}^{\min\{R, i_y\}}\chi_{\mathcal {D}^1_1}(\exp_p(ty))\cdot\hat{\sigma}(t,y)dt,
\]
where
\[
\chi_{\mathcal {D}^1_1}(x)=
\left \{
\begin{array}{lll}
&1, &x\in \mathcal {D}^1_1,\\
\\
&0, &x\notin \mathcal {D}^1_1.
\end{array}
\right.
\]
Given $y\in \mathscr{D}^1_1$, we can write
\[
\int_{r/(2\sqrt{\Lambda})}^{\min\{R, i_y\}}\chi_{\mathcal {D}^1_1}(\exp_p(ty))\cdot\hat{\sigma}(t,y)dt=\underset{j_y}{\sum}\int_{a_{j_y}}^{b_{j_y}}\hat{\sigma}(t,y)dt,
\]
where $\exp_p(b_{j_y}y)\in \Gamma$ and $\exp_p(a_{j_y}y)\in \Gamma$ if $a_{j_y}>r/(2\sqrt{\Lambda})$.

Set
\[
c_{j_y}:=\left \{
\begin{array}{lll}
&a_{j_y}, &a_{j_y}>r/(2\sqrt{\Lambda});\\
\\
&F\left(\exp^{-1}_p\left(\left(\exp_p\frac{r}{2\sqrt{\Lambda}}y\right)^*\right)\right),
&a_{j_y}=r/(2\sqrt{\Lambda}).
\end{array}
\right.
\]
Thus, $\beta\leq c_{j_y}\leq a_{j_y}$ and $\exp_p(c_{j_y}y)\in \Gamma$. Lemma \ref{le2} then yields
\begin{align*}
&\underset{j_y}{\sum}\int_{a_{j_y}}^{b_{j_y}}\hat{\sigma}(t,y)dt\leq \underset{j_y}{\sum}\int_{c_{j_y}}^{b_{j_y}}\hat{\sigma}(t,y)dt\leq \Lambda^{2n}\underset{j_y}{\sum}\frac{\V_{n,k}(b_{j_y})-\V_{n,k}(c_{j_y})}{\A_{n,k}(c_{j_y})}\hat{\sigma}(c_{j_y},y)\\
\leq & \Lambda^{2n}\frac{\V_{n,k}(R)-\V_{n,k}(\beta)}{\A_{n,k}(\beta)}\underset{j_y}{\sum}\hat{\sigma}(c_{j_y},y).
\end{align*}
The inequality above together with Lemma \ref{le3} yields
\begin{align*}
\mu(\mathcal {D}^1_1)&\leq \Lambda^{2n}\frac{\V_{n,k}(R)-\V_{n,k}(\beta)}{\A_{n,k}(\beta)}\int_{\mathscr{D}_1^1}\underset{j_y}{\sum}\hat{\sigma}(c_{j_y},y)d\nu_p(y)\\
&\leq \Lambda^{2n+\frac12}\frac{\V_{n,k}(R)-\V_{n,k}(\beta)}{\A_{n,k}(\beta)}\A_\pm(\Gamma). \tag{5.2}\label{5.2}
\end{align*}

Combining (\ref{5.1}) and (\ref{5.2}), we obtain
\begin{align*}
\frac{\A_{\pm}(\Gamma)}{\mu(D_1)}=\frac{\A_{\pm}(\Gamma)}{\mu(\mathcal {D}_1^1)}\frac{\mu(\mathcal {D}_1^1)}{\mu(D_1)}\geq \Lambda^{-(2n+\frac12)}(1-\alpha(1+\gamma))\frac{\A_{n,k}(\beta)}{\V_{n,k}(R)-\V_{n,k}(\beta)}.
\end{align*}

\noindent\textbf{Step 2}: Suppose $\mu(D_1\cap B
({r}/({2\sqrt{\Lambda}})))\geq \alpha \mu(D_1)$.

For simplicity, set
$W_i:=D_i\cap B({r}/({2\sqrt{\Lambda}}))$, $i=1,2$. Now we consider the
product space $W_1\times W_2$ with the product measure
$d\mu_{\times}:=d\mu \times d\mu$. Let
\[
N:=\{(q,w)\in W_1\times W_2:\, q\in \text{Cut}_w \text{ or } w\in \text{Cut}_q\}.
\]
Then Fubini's theorem together with \cite[Lemma 8.5.4]{BCS} yields that
\begin{align*}
\mu_{\times}(N)\leq \int_{w\in W_1}d\mu(w)\int_{q\in W_2\cap \text{Cut}_w} d\mu(q)+\int_{q\in W_2}d\mu(q)\int_{w\in W_1\cap \text{Cut}_q} d\mu(w)=0.
\end{align*}
Hence, for each $(q,w)\in (W_1\times W_2)\backslash N$, there exists a unique minimal geodesic $\gamma_{wq}$ from $w$ to $q$ with the length $L_F(\gamma_{wq})\leq r$.

We claim that $\gamma_{wq}$ is contained in $B(r)$.
In fact, if $\gamma_{wq}\cap S^+_p(r)=\{w_1,q_1\}$ (which may coincide), then
\begin{align*}
d(w,w_1)\geq d(p,w_1)-d(p,w)> \left(1-\frac{1}{2\sqrt{\Lambda}}\right)r,\ d(q_1, q)>\frac{1}{\sqrt{\Lambda}}\left(1-\frac{1}{2\sqrt{\Lambda}}\right)r.
\end{align*}
Hence,
$L_F(\gamma_{wq})\geq d(w,w_1)+d(q_1,q)>r$,
which is a contradiction!

 Since $\gamma_{wq}$ is contained in $B(r)$, it must intersect $\Gamma$. Denote by $q^*$ the last point on $\gamma_{wq}$ where $\gamma_{wq}$ intersects $\Gamma$.

%\begin{figure*}[ht] \centering
%\includegraphics[width=4cm,height=4cm]{zz.jpg}
%\caption{}
%\end{figure*}

Define
\begin{align*}
V_1&:=\{(q,w)\in W_1\times W_2-N:\, d(w,q^*)\geq d(q^*,q)\},\\
V_2&:=\{(q,w)\in W_1\times W_2-N:\, d(w,q^*)\leq d(q^*,q)\},
\end{align*}
where $q^*$ is defined as above.
Since $\mu_\times(V_1\cup V_2)=\mu_{\times}(W_1\times W_2)$, we have
\[
\mu_{\times}(V_1)\geq \frac12 \mu_\times(W_1\times W_2) \text{ or } \mu_{\times}(V_2)\geq \frac12 \mu_\times(W_1\times W_2).
\]

%\begin{figure*}[ht] \centering
%\includegraphics[width=8cm,height=4cm]{z2.jpg}
%\caption{}
%\end{figure*}

\noindent \textbf{Case I}: Suppose that $\mu_{\times}(V_1)\geq \frac12 \mu_\times(W_1\times W_2)$.

Note that
\begin{align*}
\mu_\times(V_1)=\int_{w\in W_2}d\mu\int_{\{q\in W_1:d(w,q^*)\geq d(q^*,q)\}-\text{Cut}_w} d\mu.
\end{align*}
Thus, there exist a point $w_2\in W_2$ and a measurable set $U_1\subset W_1$ such that

\noindent(1) For each $q\in U_1$, $d(w_2, q^*)\geq d(q^*, q)$ and $(q,w_2)\notin N$.

\noindent(2) $\mu(U_1)\geq \frac12 \mu(W_1)$;

Let $(t,y)$ denote the polar coordinates about $w_2$. For $q=(\rho, y)\in U_1$, set $q^*=:(\rho^*,y)$. Since
$\rho^*=d(w,q^*)\geq d(q^*,q)=\rho-\rho^*$, $\rho^*\geq \rho/2$. Set
$\rho^{**}:=\sup\{s: \exp_{w_2}(ty),\,t\in [\rho^*,s],
\text{ is contained in
} U_1-\text{Cut}_{w_2}\}$. Then $\tilde{q}:=(\rho^{**},y) \in
\overline{B({r}/({2\sqrt{\Lambda}}))}$, which implies
\[
\rho^{**}=d(w_2, \tilde{q})\leq d(w_2, p)+d(p, \tilde{q})< \frac{r}{2}+\frac{r}{2\sqrt{\Lambda}}\leq r.
\]
Since $(\tilde{q})^*=q^*$, $\rho^*\geq \rho^{**}/2$. Lemma \ref{le2} then yields
\begin{align*}
\frac{\hat{\sigma}(\rho^*,y)}{\int_{\rho^*}^{\rho^{**}}\hat{\sigma}(t,y)dt}&\geq \Lambda^{-2n}\frac{\A_{n,k}(\rho^*)}{\V_{n,k}(\rho^{**})-\V_{n,k}(\rho^{*})}\geq \Lambda^{-2n}\frac{\A_{n,k}(\rho^{**}/2)}{\V_{n,k}(\rho^{**})-\V_{n,k}(\rho^{**}/2)}\\
&\geq \Lambda^{-2n}\frac{\A_{n,k}(r/2)}{\V_{n,k}(r)-\V_{n,k}(r/2)}.
\end{align*}

Lemma \ref{le3} now yields that
\[
d \A_\pm|_{(\rho^*,y)}\geq
\Lambda^{-(2n+\frac12)}\frac{\A_{n,k}(r/2)}{\V_{n,k}(r)-\V_{n,k}(r/2)}\left({\int_{\rho^*}^{\rho^{**}}\hat{\sigma}(t,y)dt}\right)d\nu_{w_2}(y).
\]
Hence,
\[
\A_\pm(\Gamma)\geq \A_\pm(\Gamma\cap B({r}/({2\sqrt{\Lambda}})))\geq \Lambda^{-(2n+\frac12)}\frac{\A_{n,k}(r/2)}{\V_{n,k}(r)-\V_{n,k}(r/2)}\,\mu(U_1).
\]
By assumption, we have
\[
\alpha \mu(D_1)\leq \mu(D_1\cap B({r}/({2\sqrt{\Lambda}})))=\mu(W_1)\leq 2\mu(U_1),
\]
which implies
\[
\frac{\A_\pm(\Gamma)}{\mu(D_1)}\geq \frac{\alpha}{2\Lambda^{2n+\frac12}}\frac{\A_{n,k}(r/2)}{\V_{n,k}(r)-\V_{n,k}(r/2)}.
\]

\noindent \textbf{Case II}: Suppose that $\mu_{\times}(V_2)\geq \frac12 \mu_\times(W_1\times W_2)$.

Then Fubini's theorem yields that there exist a point $q_1\in W_1$ and a measurable set $U_2\subset W_2$ such that

\noindent(1) For each $w\in U_2$, $d(w, q_1^*)\leq d(q_1^*, q_1)$ and $(q_1,w)\notin N$.

\noindent(2) $\mu(U_2)\geq \frac12 \mu(W_2)$;

It should be noticeable that $q^*_1$ is dependent on the choice of $w$.
Let $w^\sharp $ denote the first point on $\gamma_{wq_1}$ where the segment intersects $\Gamma$. Thus, for each $w\in U_2$,
\[
d(w,w^\sharp )\leq d(w,q_1^*)\leq d(q_1^*,q_1)\leq d(w^\sharp ,q_1).
\]

Let $\tilde{F}$ denote the reverse of $F$. It follows from Lemma \ref{rele1} that the reverse of the geodesic $\gamma_{wq_1}$ is a minimal geodesic $\tilde{\gamma}_{q_1w}$ from $q_1$ to $w$ in $(M,\tilde{F})$. Note that $w^\sharp $ is the last point on $\tilde{\gamma}_{q_1w}$ where $\tilde{\gamma}_{q_1w}$ intersects $\Gamma$. Let $\widetilde{N}$ be defined as $N$ in $(M,\tilde{F})$. It is easy to see that $\widetilde{N}=N$. Denote by $\tilde{d}$ the metric induced by $\tilde{F}$.
Thus, $ {U}_2\subset W_2$ satisfies

\noindent(1) For each $w\in {U}_2$, $\tilde{d}(q_1,w^\sharp)\geq \tilde{d}(w^\sharp,w)$ and $(q_1,w)\notin \widetilde{N}$.

\noindent(2) $\tilde{\mu}({U}_2)\geq \frac12 \tilde{\mu
}(W_2)$;

Note that Lemma \ref{rele1} also implies that $\widetilde{\mathbf{Ric}}\geq (n-1)k$.
A similar argument to the one in Case I together with Lemma \ref{rele2} and Corollary \ref{reco} yields that
\[
\A_\mp(\Gamma)\geq \Lambda^{-(2n+\frac12)}\frac{\A_{n,k}(r/2)}{\V_{n,k}(r)-\V_{n,k}(r/2)}\mu( {U}_2).
\]
By assumption, we have
\[
\alpha \mu(D_1)\leq \mu(D_1\cap B({r}/({2\sqrt{\Lambda}})))\leq \mu(D_2\cap B({r}/({2\sqrt{\Lambda}})))=\mu(W_2)\leq 2\mu({U}_2).
\]
Hence,
\[
\frac{\A_\pm(\Gamma)}{\mu(D_1)}\geq \frac{\alpha}{2\Lambda^{2n+\frac12}}\frac{\A_{n,k}(r/2)}{\V_{n,k}(r)-\V_{n,k}(r/2)}.
\]

\noindent\textbf{Step 3}: From above, we obtain
\[
\frac{\A_\pm(\Gamma)}{\mu(D_1)}\geq  \left\{
\begin{array}{lll}
&\frac{1-\alpha(1+\Lambda^{2n}\mathscr{C})}{\Lambda^{2n+\frac12}}\mathscr{A}, &\mu(D_1\cap B ({r}/({2\sqrt{\Lambda}})))\leq \alpha \mu(D_1);\\
\\
&\frac{\alpha}{2\Lambda^{2n+\frac12}}\mathscr{B}, &\mu(D_1\cap B ({r}/({2\sqrt{\Lambda}})))\geq \alpha \mu(D_1),
\end{array}
\right.
\]
where
\[
\mathscr{A}:=\frac{\A_{n,k}(\beta)}{\V_{n,k}(R)-\V_{n,k}(\beta)},\ \mathscr{B}:=\frac{\A_{n,k}(r/2)}{\V_{n,k}(r)-\V_{n,k}(r/2)}, \ \mathscr{C}:=\frac{\V_{n,k}(R)-\V_{n,k}(r/(2\sqrt{\Lambda}))}{\V_{n,k}(r/(2\sqrt{\Lambda}))-\V_{n,k}(\beta)}.
\]
To obtain the best possible bound, we set
\[
\frac{1-\alpha(1+\Lambda^{2n}\mathscr{C})}{\Lambda^{2n+\frac12}}\mathscr{A}=\frac{\alpha}{2\Lambda^{2n+\frac12}}\mathscr{B}.
\]
Thus,
\[
\alpha=\frac{2 \mathscr{A}}{\mathscr{B}+2\mathscr{A}(1+\Lambda^{2n}\mathscr{C})}.
\]
An easy calculation then yields
\[
\frac{\A_\pm(\Gamma)}{\mu(D_1)}\geq
\frac{\A_{n,k}(\beta)\left[\V_{n,k}\left(\frac{r}{2\sqrt{\Lambda}}\right)-\V_{n,k}(\beta)\right]}{2\Lambda^{4n+\frac12}\V_{n,k}(r)\V_{n,k}(R)}.
\]
\end{proof}

Since $\frac{\A_\pm(\Gamma)}{\min\{\mu(D_1),\mu(D_2)\}}\geq\frac{\A_\pm(\Gamma)}{\mu(D_1)}$, we have the following Finslerian version of Buser's isoperimetric inequality\cite[Theorem 6.8]{Ch3}.
\begin{corollary}\label{imle}
Let $D$ be a star-like domain (with respect to $p$) in $M$
with $B^+_p(r)\subset D\subset B^+_p(R)$.  If
$\mathbf{Ric}\geq (n-1)k$ ($k<0$) and $\Lambda_F\leq \Lambda$, then
\begin{align*}
h(D)&:=\underset{\Gamma}{\inf}\frac{\min\{\A_\pm(\Gamma)\}}{\min\{\mu(D_1),\mu(D_2)\}}\\
&\geq \underset{0<\beta<\frac{r}{2\sqrt{\Lambda}}}{\max}\left\{\frac{\A_{n,k}(\beta)\left[\V_{n,k}\left(\frac{r}{2\sqrt{\Lambda}}\right)-\V_{n,k}(\beta)\right]}{2\Lambda^{4n+\frac12}\V_{n,k}(r)\V_{n,k}(R)}\right\},
\end{align*}
where $\Gamma$ varies over smooth hypersurfaces
 in $D$ satisfying

\noindent (1) $\overline{\Gamma}$ is embedded in $\overline{D}$;

\noindent (2) $\Gamma$ divides $D$ into disjoint open sets $D_1$, $D_2$ in $D$ with common boundary $\partial D_1=\partial D_2=\Gamma$.
\end{corollary}

\section{A Buser type inequality for Finsler manifolds}

Let $(\phi,\varphi):=\int_M\phi\varphi d\mu$. Then we have the following minimax principle.

\begin{lemma}\label{Coule}
Let $(M,F,d\mu)$ be a closed Finsler manifold with the reversibility $\lambda_F$ and let $D_1$, $D_2$ be pairwise disjoint normal domains (i.e., with compact closures and nonempty piecewise $C^\infty$ boundary) in $M$. Then
\[
\lambda_1(M)\leq\lambda^2_F \max\{\lambda_1(D_1),\, \lambda_1(D_2)\}.
\]
\end{lemma}
\begin{proof}
Suppose $\lambda_1(D_1)\leq \lambda_1(D_2)$. And let $\psi_i$ be the eigenfunction corresponding to $\lambda_1(D_i)$, $i=1,2$. We extend $\psi_i$ to $M$ by letting $\psi_i\equiv0$ on $M\backslash D_i$.

There exists $\alpha_1$, $\alpha_2$, not all equal to zero, satisfying
\[
\alpha_1(\psi_1,\phi)+\alpha_2(\psi_2,\phi)=0,
\]
where $\phi$ is a nonzero constant function on $M$. Set $f:=\alpha_1 \psi_1 +\alpha_2 \psi_2$. Thus, $(f,\phi)=0$ and $f\in \mathscr{H}_0(M)$. Hence,
\[
\lambda_1(M)\leq\frac{\int_MF^{*2}(df)d\mu}{\int_M f^2d\mu}\leq\lambda^2_F\frac{\int_{D_1}\alpha_1^2 F^{*2}(d\psi_1)d\mu+\int_{D_2}\alpha_2^2 F^{*2}(d\psi_2)d\mu}{\int_M (\alpha_1^2 \psi_1^2+\alpha_2^2 \psi_2^2) d\mu}\leq \lambda^2_F\lambda_1(D_2).
\]

\end{proof}

The following lemma is clear.
\begin{lemma}\label{conlem}
Let $(M,F,d\mu)$ be a closed Finsler manifold. Given a positive constant $C>0$, define a new Finsler metric $\widehat{F}$ by $\widehat{F}(y):=C^{\frac12} F(y)$. Then
\[
\widehat{\mathbf{Ric}}(y)=\frac1C \mathbf{Ric},\ \widehat{\lambda}_1(M)=\frac{1}{{C}}\lambda_1(M),\ \widehat{\mathbbm{h}}(M)=\frac{1}{\sqrt{C}} {\mathbbm{h}}(M).
\]
\end{lemma}

\begin{proof}[\textbf{Proof of Theorem \ref{Th2}}]By Lemma \ref{conlem}, we can suppose that $\mathbf{Ric}\geq -(n-1)$, i.e., $\delta=1$.
Given any $\epsilon>0$, let $\Gamma$, $D_1$ and $D_2$ be as in Definition \ref{def2} such that $0\leq\mathscr{I}-{\mathbbm{h}}(M)<\epsilon$, where
\[
\mathscr{I}:=\frac{\min\{\A_\pm(\Gamma)\}}{\min\{\mu(D_1),\mu(D_2)\}}.
\]

\noindent \textbf{Step 1.}
For $n\geq 3$, $\Gamma$ may satisfy $\max_{p\in M}d(p,\Gamma)<\rho$, for any $\rho>0$ ("the problem of hair"\cite{B1}). Hence, we will find a new set $\tilde{\Gamma}$ to replace $\Gamma$.

Let $\mathscr{P}:=\{p_1,\ldots,p_k\}$ be a forward complete $r$-package in $M$, that is

\noindent(1) $d(p_i,p_j)\geq 2r$, for $i\neq j$;

\noindent(2) ${\cup}_{1\leq i\leq k}B^+_{p_i}(2r\sqrt{\Lambda})=M$.

For each $p_i\in \mathscr{P}$, the Dirichlet region of $p_i$ is defined by
\[
\mathscr{D}_i:=\{q\in M: \, d(p_i,q)\leq d(p_j,q), \text{for all }1\leq j\leq k\}.
\]
Property (1) and (2) imply that
\[
B^+_{p_i}(r/\sqrt{\Lambda})\subset \mathscr{D}_i \subset B^+_{p_i}(2r\sqrt{\Lambda}).
\]

 Lemma \ref{le2} yields
\[
\mu(\mathscr{D}_i)\leq \mu(B^+_{p_i}(2r\sqrt{\Lambda}))\leq \Lambda^{2n}\frac{\V_{n,-1}(2r\sqrt{\Lambda})}{\V_{n,-1}(r/(2\Lambda))}\mu(B^+_{p_i}(r/(2\Lambda))),\tag{6.1}\label{*1}
\]

Now it follows from Corollary \ref{imle} that for $0<r<\frac{1}{2\sqrt{\Lambda}}$,
\[
h(\mathscr{D}_i)\geq \frac{\A_{n,-1}\left(\frac{r}{4{\Lambda}}\right)\V_{n,-1}\left(\frac{r}{4{\Lambda}}\right)}{4\Lambda^{4n+\frac12}\V_{n,-1}\left(\frac{r}{\sqrt{\Lambda}}\right)\V_{n,-1}(2r\sqrt{\Lambda})}\geq \frac{1}{C_1(n,\Lambda)\,r}=:\mathscr{J}(r).\tag{6.2}\label{Es1}
\]
Here, $C_1(n,\Lambda)$ is a positive constant depending only on $n$ and $\Lambda$.
One can easily check that
\[
\frac{\mathscr{I}}{\mathscr{J}(r)}\left[\Lambda^{2n}\frac{\V_{n,-1}(4r\sqrt{\Lambda})}{\V_{n,-1}(r/(2\Lambda))}\right]<\frac18,\tag{6.3}\label{**}
\]
for
\[
0<r<\min\left\{\frac{1}{4\sqrt{\Lambda}},\frac{1}{C_2(n,\Lambda)\,\mathscr{I}}\right\}.
\]

\noindent\textbf{Claim}: For $i\neq j$,
\[
\mathscr{D}_i\cap \mathscr{D}_j=\{q\in M: d(p_i,q)=d(p_j,q)\leq d(p_s,q), \text{ for all } 1\leq s \leq k\}
\]
has measure zero (with respect to $d\mu$).

Note that $d(p_i,\cdot)$ and $d(p_j,\cdot)$ are smooth
on $\text{int}(\mathscr{D}_i\cap \mathscr{D}_j)-\text{Cut}_{p_i}\cup
\text{Cut}_{p_j}$. If $\nabla(d(p_i,x)-d(p_j,x))=0$ for some $x\in \text{int}(
\mathscr{D}_i\cap \mathscr{D}_j)-\text{Cut}_{p_i}\cup
\text{Cut}_{p_j}$, then $d(d(p_i,x)-d(p_j,x))=0$, that is,
$d(d(p_i,x))=d(d(p_j,x))$ and
\[
\nabla d(p_i,x)=\nabla d(p_j,x).
\]
This implies that the unique minimal geodesic from $p_i$ to $x$
overlaps the unique minimal geodesic from $p_j$ to $x$. Since
$d(p_i,x)=d(p_j,x)$, we have $p_i=p_j$, which is a contradiction!
Hence, $d(p_i,x)-d(p_j,x)$ is regular on $\text{int}(\mathscr{D}_i\cap
\mathscr{D}_j)-\text{Cut}_{p_i}\cup \text{Cut}_{p_j}$, which implies $
\text{dim}(\text{int}(\mathscr{D}_i\cap
\mathscr{D}_j)-\text{Cut}_{p_i}\cup \text{Cut}_{p_j})\leq (n-1)$. Therefore, $\text{int}(\mathscr{D}_i\cap
\mathscr{D}_j)=\emptyset$
and $\mathscr{D}_i\cap \mathscr{D}_j$ has measure zero. Thus,
\[
\mu(\mathscr{D}_i\cup \mathscr{D}_j)=\mu(\mathscr{D}_i)+\mu(\mathscr{D}_j).\tag{6.4}\label{*3}
\]

Enumerate the collection $\{p_i\}_{i=1}^k$ in such way that
\begin{align*}
&\mu\left(D_1\cap B^+_{p_i}\left(\frac{r}{2\Lambda}\right)\right)\leq \frac12 \mu\left(B^+_{p_i}\left(\frac{r}{2\Lambda}\right)\right), \text{ for } i=1,\cdots, m;\\
&\mu\left(D_1\cap B^+_{p_i}\left(\frac{r}{2\Lambda}\right)\right)> \frac12 \mu\left(B^+_{p_i}\left(\frac{r}{2\Lambda}\right)\right), \text{ for } i=m+1,\cdots, k.
\end{align*}
Lemma \ref{infact} implies that for $i=1,\cdots,m$,
\begin{align*}
\A_\pm(\Gamma\cap \text{int}(\mathscr{D}_i))\geq \mathscr{J}(r) \mu(D_1\cap \mathscr{D}_i).\tag{6.5}\label{*2}
\end{align*}
From (\ref{*1}), (\ref{*3}) and (\ref{*2}), we obtain
\begin{align*}
\overset{m}{\underset{i=1}{\sum}}\mu(D_1\cap \mathscr{D}_i)
\leq \frac{1}{\mathscr{J}(r)}\A_\pm(\Gamma),
\end{align*}
which together with (\ref{**}) yields that
\begin{align*}
&\overset{m}{\underset{i=1}{\sum}}\mu(D_1\cap \mathscr{D}_i)\leq
\frac{1}{\mathscr{J}(r)}\min\{\A_\pm(\Gamma)\}\\ =
&\frac{\mathscr{I}}{\mathscr{J}(r)}\min\{\mu(D_1),\mu(D_2)\} <
\frac14\Lambda^{-2n}\frac{\V_{n,-1}(r/(2\Lambda))}{\V_{n,-1}(2r\sqrt{\Lambda})}\mu(D_1).\tag{6.6}\label{*4}
\end{align*}

Note that $\{\mathscr{D}_i\}_{i=1}^k$ is a closed covering of $M$. Now we claim that
there exists a point $p_i\in \mathscr{P}$ such that
\[
\mu(D_1\cap B^+_{p_i}(r/(2\Lambda)))>\frac12\mu(B^+_{p_i}(r/(2\Lambda))).
\]
If not, then (\ref{*4}) together with (\ref{*3}) yields 
\begin{align*}
\mu(D_1)=\overset{k}{\underset{i=1}{\sum}}\mu(D_1\cap \mathscr{D}_i)<
\frac14\mu(D_1),
\end{align*}
which is a contradiction.

Likewise, there also exists a point $p_j\in \mathscr{P}$ such that
\[
\mu(D_2\cap B^+_{p_j}(r/(2\Lambda)))>\frac12\mu(B^+_{p_j}(r/(2\Lambda))).
\]
Thus, the following sets are not empty:
\begin{align*}
\tilde{D}_1&:=\left\{q\in M:\, \mu(D_1\cap B^+_q(r/(2\Lambda)))>\frac12\mu (B^+_q(r/(2\Lambda)))\right\};\\
\tilde{D}_2&:=\left\{q\in M:\, \mu(D_2\cap B^+_q(r/(2\Lambda)))>\frac12\mu (B^+_q(r/(2\Lambda)))\right\}.
\end{align*}

Since the continuity of the map $q\mapsto \mu(D_1\cap B^+_q(r/(2\Lambda)))-\mu(D_2\cap B^+_q(r/(2\Lambda)))$, the open submanifolds $\tilde{D}_1$ and $\tilde{D}_2$ are separated by the closed subset
\[
\tilde{\Gamma}:=\{q\in M: \mu(D_1\cap B^+_q(r/(2\Lambda)))=\mu(D_2\cap B^+_q(r/(2\Lambda)))\, \}.\tag{6.7}\label{*5}
\]

\noindent \textbf{Step 2.}
Define
\[
\tilde{\Gamma}^t:=\{q\in M:\, d(\tilde{\Gamma}, q)\leq t\}.
\]
Now choose a new forward complete $r$-package $\mathscr{Q}=\{q_1,\ldots, q_l\}$ in $M$ such that:

\noindent (1) $q_1,\ldots,q_s\in \tilde{\Gamma}$ and
$\tilde{\Gamma}\subset\cup_{1\leq i \leq
s}B^+_{q_i}(2r\sqrt{\Lambda})$;

\noindent (2) $q_{s+1},\ldots, q_{m}\in \tilde{D}_1$ and $q_{m+1},\ldots, q_l\in \tilde{D}_2$.

Since $\tilde{\Gamma}^t\subset \cup_{1\leq i\leq s} B^+_{q_i}(2r\sqrt{\Lambda}+t)$ and $q_1,\ldots,q_s\in \tilde{\Gamma}$, by (\ref{*5}) and Lemma \ref{le2}, we have
\begin{align*}
\mu(\tilde{\Gamma}^t)&\leq \mu(\overset{s}{\underset{i=1}{\cup}} B^+_{q_i}(2r\sqrt{\Lambda}+t))\leq \overset{s}{\underset{i=1}{\sum}}\mu(B^+_{q_i}(2r\sqrt{\Lambda}+t))\\
&\leq\Lambda^{2n} \frac{\V_{n,-1}(2r\sqrt{\Lambda}+t)}{\V_{n,-1}(r/(2\Lambda))}\overset{s}{\underset{i=1}{\sum}}\mu(B^+_{q_i}(r/(2\Lambda)))\\
&=2\Lambda^{2n} \frac{\V_{n,-1}(2r\sqrt{\Lambda}+t)}{\V_{n,-1}(r/(2\Lambda))}\overset{s}{\underset{i=1}{\sum}}\mu(D_1\cap B^+_{q_i}(r/(2\Lambda))).
\end{align*}
It follows from Corollary \ref{imle} that for $1\leq i \leq s$,
\[
\frac{\A_\pm(\Gamma\cap B^+_{q_i}(r/(2\Lambda)))}{\mu(D_1\cap B^+_{q_i}(r/(2\Lambda)) )}\geq h(B^+_{q_i}(r/(2\Lambda)))\geq \mathscr{J}(r/(2\sqrt{\Lambda}))\geq \mathscr{J}(r).
\]
Since $d(q_i,q_j)\geq2r$, $B^+_{q_i}(r/(2\Lambda))\cap B^+_{q_j}(r/(2\Lambda))=\emptyset$.  Hence, we have
\begin{align*}
\mu(\tilde{\Gamma}^t)
\leq 2\Lambda^{2n} \frac{\V_{n,-1}(2r\sqrt{\Lambda}+t)}{\mathscr{J}(r)\V_{n,-1}(r/(2\Lambda))}\A_\pm(\Gamma),
\end{align*}
which implies
\begin{align*}
\mu(\tilde{\Gamma}^t)&\leq 2\Lambda^{2n}\frac{\V_{n,-1}(2r\sqrt{\Lambda}+t)}{\mathscr{J}(r)\V_{n,-1}(r/(2\Lambda))}\min\{\A_\pm(\Gamma)\}\\
&=2\Lambda^{2n}\frac{\mathscr{I}\V_{n,-1}(2r\sqrt{\Lambda}+t)}{\mathscr{J}(r)\V_{n,-1}(r/(2\Lambda))}\min\{\mu(D_1),\mu(D_2)\}\\
&\leq 2\Lambda^{2n}\frac{\mathscr{I}\V_{n,-1}(2r\sqrt{\Lambda}+t)}{\mathscr{J}(r)\V_{n,-1}(r/(2\Lambda))}\mu(D_1). \tag{6.8}\label{*6}
\end{align*}

Now let $t=2r\sqrt{\Lambda}$.

\noindent \textbf{Claim}:
If $q\in \tilde{D}_2-\tilde{\Gamma}^{2r{\sqrt\Lambda}}$, then $q$ must be contained in some Dirichlet region $\mathscr{D}_i$, $m+1\leq i \leq l$.

Choose a point $x\in (\tilde{D}_1\cup \tilde{\Gamma})$. Suppose that the minimal
geodesic from $x$ to $q$ intersects $\tilde{\Gamma}$ at $y$. Thus,
\begin{align*}
d(x,q)\geq d(y,q)\geq d(\tilde{\Gamma}, q)>2r{\sqrt{\Lambda}},
\end{align*}
which implies that $d(q_i,q)>2r{\sqrt{\Lambda}}$, $1\leq i \leq m$. Since $\cup_{1\leq i\leq l}B^+_{q_i}(2r\sqrt{\Lambda})\supset M$, there exists $i\in \{m+1,\ldots, l\}$ such that $q\in B^+_{q_{i}}(2r\sqrt{\Lambda})$. Choose $i_0\in \{m+1,\ldots, l\}$ such that
\[
d(q_{i_0}, q)=\underset{m+1\leq i\leq
l}{\min}d(q_i,q)<2r\sqrt{\Lambda}.
\]
Then $q\in \mathscr{D}_{i_0}$. The claim is true. Hence, $\tilde{D}_2-\tilde{\Gamma}^{2r\sqrt{\Lambda}}\subset \cup_{m+1\leq i\leq l}\mathscr{D}_i$.

By (\ref{*4}), (\ref{*6}) and (\ref{**}), we have
\begin{align*}
\mu(\tilde{D}_1-\tilde{\Gamma}^{2r\sqrt{\Lambda}})&\geq \mu(D_1\cap (\tilde{D}_1-\tilde{\Gamma}^{2r\sqrt{\Lambda}}))\\
&= \mu(D_1)-\mu(D_1\cap (\tilde{D}_2-\tilde{\Gamma}^{2r\sqrt{\Lambda}}))-\mu(D_1\cap \tilde{\Gamma}^{2r\sqrt{\Lambda}})\\
&\geq \mu(D_1)-\overset{l}{\underset{i=m+1}{\sum}}\mu(D_1\cap \mathscr{D}_i)-\mu(\tilde{\Gamma}^{2r\sqrt{\Lambda}})\\
&\geq \mu(D_1)-\frac{1}{4}\mu(D_1)-\frac14 \mu(D_1)=\frac12\mu(D_1)>0.\tag{6.9}\label{*7}
\end{align*}

Lemma \ref{Coule} yields
\[
\lambda_1(M)\leq {\Lambda}\max\{\lambda_1(\tilde{D}_1),\lambda(\tilde{D}_2)\}.\tag{6.10}\label{*8}
\]
Without loss of generality, we suppose that $\lambda_1(\tilde{D}_1)\geq \lambda_1(\tilde{D}_2)$. Now, we estimate $\lambda_1(\tilde{D}_1)$.
Define a function on $\tilde{D}_1$ by
\[
f(q):=
\left \{
\begin{array}{lll}
&\frac{d(\tilde{\Gamma},q)}{2r\sqrt{\Lambda}}, &q\in \tilde{D}_1\cap \tilde{\Gamma}^{2r\sqrt{\Lambda}} \\
\\
&1,&q\in \tilde{D}_1-\tilde{\Gamma}^{2r\sqrt{\Lambda}}.
\end{array}
\right.
\]
Clearly,
\[
F^{*2}(df(q))=
\left \{
\begin{array}{lll}
&\frac{1}{4r^2{\Lambda}}, &q\in \tilde{D}_1\cap \tilde{\Gamma}^{2r\sqrt{\Lambda}} \\
\\
&0,&q\in \tilde{D}_1-\tilde{\Gamma}^{2r\sqrt{\Lambda}}.
\end{array}
\right.
\]
 (\ref{*6}) together with (\ref{*7}) yields that
\begin{align*}
&\int_{\tilde{D}_1}F^{*2}(df)d\mu\leq \frac{1}{4r^2{\Lambda}}\mu(\tilde{\Gamma}^{2r\sqrt{\Lambda}})\leq \Lambda^{2n-1}\frac{\mathscr{I}\V_{n,-1}(4r\sqrt{\Lambda})}{2r^2\mathscr{J}(r)\V_{n,-1}(r/(2\Lambda))}\mu(D_1),\\
&\int_{\tilde{D}_1}f^2d\mu\geq \mu(\tilde{D}_1)-\mu(\tilde{\Gamma}^{2r\sqrt{\Lambda}})\geq \frac12\mu(D_1).
\end{align*}
Thus, by (\ref{Es1}), we obtain
\begin{align*}
\lambda_1(\tilde{D}_1)&\leq \frac{\int_{\tilde{D}_1}F^{*2}(df)d\mu}{\int_{\tilde{D}_1}f^2d\mu}\leq \Lambda^{2n-1}\frac{\mathscr{I}\V_{n,-1}(4r\sqrt{\Lambda})}{r^2\mathscr{J}(r)\V_{n,-1}(r/(2\Lambda))}\\
&\leq \frac{C_3(n,\Lambda)\,\mathscr{I}}{r},
\end{align*}
for
\[
0<r<\min\left\{\frac{1}{4\sqrt{\Lambda}},\frac{1}{C_2(n,\Lambda)\,\mathscr{I}}\right\}.
\]

Hence, (\ref{*8}) yields that
\[
\lambda_1(M)\leq C_4(n,\Lambda)(\mathscr{I}+\mathscr{I}^2).
\]
We are done by letting $\epsilon\rightarrow 0^+$.
\end{proof}

For Randers metrics, we have the following corollary.
\begin{corollary}\label{Randers1}
Let $(M,\alpha+\beta)$ be a $n$-dimensional closed Randers manifold. Then
\[
\Lambda_F = \frac{(1+b)^2}{(1-b)^2}=\lambda^2_F,
\]
where $b:=\sup_{x\in M}\|\beta\|_\alpha$. Hence,
$\lambda_1(M)\leq C(n,b)\left(\delta
{\mathbbm{h}}(M)+{\mathbbm{h}^2}(M)\right)$.
\end{corollary}
\begin{proof}Since $M$ is closed, there exists a point $x\in M$, such that $\|\beta\|_\alpha(x)=b$.
Choose $y$, $X\in T_xM$ with $\|X\|_\alpha=\|y\|_\alpha=1$. For convenience, we set $y=s X+X_1^\bot$ and $\beta=t X+ X_2^\bot$. Here, we view $\beta$ as a tangent vector in $(T_xM, \alpha)$ and $\langle X, X_i^\bot\rangle_\alpha=0$, $i=1,2$.
By \cite[(1.6)]{Sh1}, one has
\[
g_y(X,X)=[1+\beta(y)](1-s^2)+(s+t)^2,\ -1\leq s\leq 1,\ -b\leq t\leq b.
\]
Clearly, $(1-b)^2\leq g_y(X,X)$, with equality if and only if $y=\pm X$ and $\beta=\mp bX$, and $g_y(X,X)\leq (1+b)^2$, with equality if and only if $y=\pm X$ and $\beta= \pm bX$. Hence, $\Lambda_F={(1+b)^2}{(1-b)^{-2}}$.
\end{proof}
\begin{remark}Given a Randers metric $F=\alpha+\beta$, the Holmes-Thompson measure $d\mu_{HT}=dV_\alpha$, where $dV_\alpha$ is the Riemannian measure induced by $\alpha$. By \cite[Example 3.2.1]{Sh1}, one can show
\[
\frac{\lambda_1(M,\alpha)}{(1+b)^2}\leq \lambda_1(M,F)\leq\frac{\lambda_1(M,\alpha)}{(1-b)^2},
\]
where $\lambda_1(M,\alpha)$ (resp. $\lambda_1(M,F)$) is the first eigenvalue of $(M,\alpha)$ (resp. $(M,F,d\mu_{HT})$).
\end{remark}

By \cite{Sh2,ZY}, we can see that the upper bound for the uniform constant in Lemma \ref{le2} can be replaced by the lower bound for the S-curvature. Using the similar argument, one can show the following theorem.
\begin{theorem}\label{Sone}
Let $(M,F,d\mu)$ be a closed Finsler $n$-manifold with the Ricci
curvature $\mathbf{Ric}\geq -(n-1)\delta^2$, the S-curavture
$\mathbf{S}\geq (n-1)\eta$ and the reversibility $\lambda_F\leq
\lambda$.
Then
\[
\lambda_1(M)\leq C(n,\lambda,\eta)\left(\delta
{\mathbbm{h}}(M)+{\mathbbm{h}^2}(M)\right).
\]
\end{theorem}

It follows from \cite{Sh1} that for the Busemann-Hausdorff measure, the S-curvature of a Berwald manifold always vanishes. Furthermore, we have the following
\begin{theorem}\label{Berwald1}
For the Holmes-Thompson measure, the S-curvature of a Berwald manifold also vanishes.
\end{theorem}
\begin{proof}Let $(M,F)$ be a $n$-dimensional Berwald manifold and let $\gamma_y(t)$ be a unit speed geodesic with $\dot{\gamma}_y(0)=y$. Denote by $P_t$ the parallel transportation along $\gamma_y(t)$. Choose a basis $\{e_i\}$ of $T_{\gamma_y(0)}M$. Then $E_i(t):=P_t e_i$, $1\leq i \leq n$, is a basis of $T_{\gamma_y(t)}M$. Let $(y^i)$ (resp. $(z^i)$) denote the corresponding coordinate system in $T_{\gamma_y(0)}M$ (resp. $T_{\gamma_y(t)}M$). Thus, $z^i\circ P_t=y^i$.

For any $w\in S_{\gamma_y(0)}M$, we have
\begin{align*}
\frac{d}{dt} g_{(\gamma_y(t), P_tw)}(E_i(t),E_j(t))=\frac{2}{F(P_tw)}A_{(\gamma_y(t), P_tw)}\left(E_i(t),E_j(t),\nabla_{\dot{\gamma}_y} P_t w\right)=0.
\end{align*}
Note that $P_t(B_{\gamma_y(0)}M)=B_{\gamma_y(t)}M$, where $B_xM:=\{y\in T_xM: F(x,y)<1\}$. The equation above together with \cite[Lemma 5.3.2]{Sh1} yields that
\begin{align*}
&\int_{v\in B_{\gamma_y(t)}M}\det g_{(\gamma_y(t),v)}(E_i(t),E_j(t))dz^1\wedge\cdots \wedge dz^n\\
=&\int_{w\in B_{\gamma_y(0)}M}\det g_{(\gamma_y(t),P_t w)}(P_t e_i,P_t e_j)P_t^*dz^1\wedge\cdots \wedge P_t^*dz^n\\
=&\int_{w\in B_{\gamma_y(0)}M}\det g_{(\gamma_y(0),w)}(e_i,e_j)dy^1\wedge\cdots \wedge dy^n.
\end{align*}
Thus, $\tau_{HT}(\dot{\gamma_y}(t))=\tau_{HT}(\dot{\gamma_y}(0))$,
which implies that $\mathbf{S}_{HT}\equiv0$.
\end{proof}

Theorem \ref{Sone} together with Theorem \ref{Berwald1} now yields the following

\begin{corollary}
Let $(M,F,d\mu)$ be a $n$-dimensional closed Berwald manifold with the Ricci
curvature $\mathbf{Ric}\geq -(n-1)\delta^2$ and the reversibility $\lambda_F\leq
\lambda$. Then
we have
\[
\lambda_1(M)\leq C(n,\lambda)\left(\delta
{\mathbbm{h}}(M)+{\mathbbm{h}^2}(M)\right).
\]
\end{corollary}

\end{document}